\documentclass[11pt]{article}

\usepackage{amsmath,amssymb, amsthm}
\usepackage{graphicx, enumerate}

\newtheorem{thm}{Theorem}
\newtheorem{lem}{Lemma}

\newtheorem{conj}{Conjecture}

\newcounter{counter}

\def\less{\backslash}

\begin{document}
\title{On Hypergraph Lagrangians and Frankl-F\"uredi's Conjecture}
\author{
Hui Lei
\thanks{Center for Combinatorics and LPMC, Nankai University, Tianjin 300071, China,
   ({\tt hlei@mail.nankai.edu.cn}). This author was supported in part by National Natural Science Foundation of China (No. 11771221).}
\and
Linyuan Lu
\thanks{University of South Carolina, Columbia, SC 29208,
({\tt lu@math.sc.edu}). This author was supported in part by NSF
grant DMS 1600811.}
}
\maketitle

\begin{abstract}
Frankl and F\"uredi conjectured in 1989 that the maximum Lagrangian, denoted by $\lambda_r(m)$, among all $r$-uniform hypergraphs of fixed size $m$ is achieved by the minimum hypergraph $C_{r,m}$ under the colexicographic order.
We say $m$ in {\em principal domain} if there exists an integer $t$ such that
${t-1\choose r}\leq m\leq {t\choose r}-{t-2\choose r-2}$. If $m$ is in the principal domain,
then Frankl-F\"uredi's conjecture has a very simple expression:
$$\lambda_r(m)=\frac{1}{(t-1)^r}{t-1\choose r}.$$
Many previous results are focusing on $r=3$. For $r\geq 4$,
Tyomkyn in 2017 proved that Frankl-F\"{u}redi's conjecture holds whenever ${t-1\choose r} \leq m \leq
{t\choose r} -{t-2\choose r-2}- \delta_rt^{r-2}$ for a constant $\delta_r>0$. In this paper, we improve
Tyomkyn's result by showing Frankl-F\"{u}redi's conjecture holds whenever ${t-1\choose r} \leq m \leq
{t\choose r} -{t-2\choose r-2}- \delta_r't^{r-\frac{7}{3}}$ for a constant $\delta_r'>0$.
\end{abstract}

\section{Introduction}
The Lagrangians of hypergraphs are closely related to the Tur\'an densities in the extremal hypergraph theory.
Given an $r$-uniform hypergraph $H$ on a vertex set $[n]:=\{1,2,\ldots,n\}$,  the Lagrangian of $H$, denoted by $\lambda(H)$,  is defined to  be
$$\lambda(H)=\max_{{\bf x}\in \mathbb{R}^n_+: \|{\bf x}\|_1=1}\!\!\!
\sum_{\{i_1,i_2,\ldots, i_r\}\in E(H)} x_{i_1}x_{i_2}\cdots x_{i_r},$$
where the maximum is taken over on a simplex $\{x\in \mathbb{R}^n \colon x_1,\ldots, x_n\geq 0,
\mbox{ and } \sum_{i=1}^n x_i=1\}$. A maximum point ${\vec x}_0$ is called an {\em optimal
legal weighting} and the set of its nonzero coordinates in ${\vec x}_0$ is called a {\em support}
of $H$ (see section 2 for details.)
 One can show that $r! \cdot \lambda(H)$ is the supremum of
edge densities among all hypergraphs which are blow-ups of $H$. It has important applications in the Tur\'an theory.

The concept of Lagrangians of graphs was introduced by Motzkin and Straus \cite{MS1965} in 1965, who proved that $\lambda(H)=\frac{1}{2}\left(1-\frac{1}{\omega(H)}\right)$,
where $\omega(H)$ is the clique number of a graph $H$. Their theorem implies Tur\'an's theorem.

 Let $\lambda_r(m)$ be the maximum of Lagrangians among all $r$-uniform hypergraphs with $m$ edges.
 Then Motzkin-Straus' result implies $\lambda_2(m)=\frac{1}{2}\left(1-\frac{1}{t-1}\right)$ if ${t-1\choose 2}\leq m<{t\choose 2}$
 for some integer $t$.

 For any $r\geq 2$ and any $m\geq 1$, let $C_{r,m}$ be the
 $r$-uniform hypergraph consisting of the first $m$ sets in ${\mathbb{N} \choose r}$ in the colexicographic order (that is $A<B$ if $\max(A\Delta B)\in B$.)
 For example, for $r=3$, the first 5 triple sets under the 
 colexicographic order are given below:
 $$\{1,2,3\}<\{1,2,4\}<\{1,3,4\}<\{2,3,4\}<\{1,2,5\}<\cdots$$
 so $C_{3,5}$ is the hypergraph on $5$ vertices with the following edge set
 $$E(C_{3,5})=\{\{1,2,3\},\{1,2,4\},\{1,3,4\},\{2,3,4\},\{1,2,5\}\}.$$

 One can easy to check that if $m={t\choose r}$ for some integer $t$, then
 $C_{r,m}$ is just the complete $r$-uniform hypergraph $K^r_t$ (or $[t]^{(r)}$
 under Talbot's notation \cite{T2002}.)

 In 1989, Frankl and F\"{u}redi made the following conjecture:
 \begin{conj}[{\cite{FF1989}}] \label{conjecture}
   For all $r\geq 3$ and any $m\geq 1$, we have $\lambda_r(m)\leq \lambda(C_{r,m}).$
 \end{conj}

 Talbot \cite{T2002} pointed out that $\lambda(C_{r,m})$ remains  a constant
 ($\lambda(C_{r,m})\equiv{t-1\choose r}/(t-1)^r$)
 for  $m\in [{t-1\choose r}, {t\choose r}-{t-2\choose r-2}]$, and jumps for every
 $m\in [{t\choose r}-{t-2\choose r-2}, {t\choose r}]$.
   Tyomkyn called $m={t\choose r}$
   the principal case. Here we refer the interval $[{t-1\choose r}, {t\choose r}-{t-2\choose r-2}]$ as the {\em principal domain}, and refer $m={t\choose r}-{t-2\choose r-2}$, the {\em critical case}. Most partial results on Frankl-F\"{u}redi's conjecture occur in the principal domain.

   For $r=3$, Talbot \cite{T2002} proved that Conjecture \ref{conjecture} holds whenever ${{t-1}\choose3}\leq m\leq {t\choose 3}-{{t-2}\choose 1}-(t-1)={t\choose3}-(2t-3)$ for some $t\in \mathbb{N}$.
   Tang, Peng, Zhang and Zhao \cite{TPZZ2016}
   extended the interval to ${{t-1}\choose3}\leq m\leq {t\choose 3}-{{t-2}\choose 1}-\frac 1 2(t-1)$.
   Recently, Tyomkyn \cite{T2017} further extended the interval to
   $[{{t-1}\choose 3}, {t\choose 3}- {t-2\choose 1}-\delta_3t^{3/4}]$ for some constant $\delta_3>0$.
   These results can be rephrased in term of the gap (i.e.  the number of missed values)
   in the principal domain: the gap drops from $t-1$, to $\frac{1}{2}(t-1)$, and further to $O(t^{3/4})$. Recently,  Lei, Lu, and Peng \cite{LLP} further reduced the gap to $O(t^{2/3})$.

   For $r\geq 4$, there are fewer results than at $r=3$. In 2017, Tyomkyn \cite{T2017} proved
   the following theorem.
 \begin{thm}[\cite{T2017}]\label{3/4}
 For $r\geq 4$,  there exists a constant $\delta_r>0$ such that
    for any $m$ satisfying ${{t-1}\choose r}\leq m\leq {t\choose r}-{t-2\choose r-2}-\delta_rt^{r-2}$ we have
    $$\lambda_r(m)=\frac{{t-1\choose r}}{(t-1)^r}.$$
 \end{thm}
Here is our main theorem.
  \begin{thm}\label{mainthm}
 For $r\geq 4$,  there exists a constant $\delta_r>0$ such that
    for any $m$ satisfying ${{t-1}\choose r}\leq m\leq {t\choose r}-{t-2\choose r-2}-\delta_rt^{r-\frac{7}{3}}$ we have
    $$\lambda_r(m)=\frac{{t-1\choose r}}{(t-1)^r}.$$
 \end{thm}

Tyomkyn \cite{T2017} proved that the gap can be reduced to $O(t^{r-9/4})$ under an assumption that the hypergraphs have support on $t$ vertices.
We actually proved that the maximum hypergraphs have support on $t$ vertices
for sufficiently large $t$ (see Lemma \ref{keylemma}.)
Moreover, our gap $O(t^{r-7/3})$ improves $O(t^{r-9/4})$ on the exponent slightly.

Another related result is a smooth upper bound on $\lambda_r(m)$. The following result,
which was conjectured (and partially solved for $r=3,4,5$ and any $m$; and for the case $m\geq {4(r-1)(r-2)\choose r}$) by Nikiforov \cite{Nikiforov2018},
was completely proved by the second author.

\begin{thm}[{\cite{maxPspec}}] \label{smooth}
 For all $r\geq 2$ and all $m\geq 1$, if we write $m = {s\choose r}$
   for some real number $s\geq r-1$, then have $$\lambda_r(m)\leq m s^{-r}.$$
  The equality holds if and only if $s$ is an integer and the hypergraph achieving $\lambda_r(m)$ must be the complete  $r$-uniform hypergraph $K^r_s$  (possibly with some isolated vertices added.)
 \end{thm}

The paper will be organized as follows:  the notation and previous lemmas will be given in Section 2. In Section 3, we prove a key lemma that the
maximum hypergraphs have support $t$ for $t$ sufficiently large. Finally, the proof of Theorem \ref{mainthm} will be given in Section 4.

\section{Notation and Preliminaries}
Let $\mathbb{N}$ be the set of all positive integers and $[t]$  the
set of first $t$ positive integers. For any integer $r\geq 2$ and a set $V$, we use $V^{(r)}$ (or ${V\choose r}$) to denote all $r$-subsets of $V$.
An $r$-unform hypergraph (or $r$-graph, for short) consists of a vertex set $V$ and an edge set
$E\subseteq V^{(r)}$.
Given an $r$-graph $G=(V,E)$ and a set $S\subseteq V$ with $|S|<r$, the $(r-|S|)$-uniform {\it link hypergraph} of $S$ is defined as $G_S=(V,E_S)$ with $E_S:=\{f\in V^{(r-|S|)}\colon f\cup S\in E\}$. We will denote the complement graph of $G_S$ by $G^c_S=(V,E^c_S)$ with $E^c_S:=\{f\in V^{(r-|S|)}\colon f\cup S\in V^{(r)}\less E\}$. Define $G_{i\less j}=(V, E_{i\less j})$, where $E_{i\less j}:=\{f\in E_i\less E_j\colon j\notin f\}$ and $G^c_{i\less j}=(V, E^c_{i\less j})$, where $E^c_{i\less j}:=\{f\colon f\cup\{i\}\in E^c ~\text{but}~ f\cup\{j\}\in E\}$, i.e., $E^c_{i\less j}=E^c_i\cap E_j$. Let $G-i$ be the $r$-graph obtained from $G$ by deleting vertex $i$ and those edges containing $i$. A hypergraph $G=(V,E)$ is said to {\it cover} a vertex pair $\{i,j\}$ if there exists an edge $e\in E$ with $\{i,j\}\subseteq e$. The $r$-graph $G$ is said to {\it cover pairs} if it covers every pair $\{i,j\}\subseteq V^{(2)}$.

From now on we assume that $G$ is an $r$-graph on the vertex set $[n]$.
Given a vector $\vec x=(x_1,\ldots,x_n)\in \mathcal{R}^n$,
we write $x_f=x_{i_1}x_{i_2}\cdots x_{i_r}$ if $f=\{i_1,i_2, \ldots, i_r\}$.
The {\it weight polynomial} of $G$ is given by
$$w(G,\vec x)=\sum_{e\in E(G)}x_e.$$
We call $\vec x=(x_1,\ldots,x_n)\in \mathcal{R}^n$  a {\it legal weighting} for $G$ if $x_i\geq0$ for any $i\in[n]$ and $\sum_{i=1}^n x_i=1$.
The set of all legal weightings forms a standard simplex in $\mathbb{R}^n$.
We call a legal weighting ${\vec x}_0$ {\it optimal} if
$w(G,\vec x)$ reaches the maximum at $\vec x={\vec x}_0$ in this simplex. The maximum value
of $w(G,\vec x)$, denoted by $\lambda(G)$,  is called the {\em Lagrangian} of $G$.

\begin{lem}[\cite{FF1989,T2017}]\label{Hmr}
Suppose that $G\subseteq [n]^{(r)}$ and $\overset{\rightarrow}{x}=(x_1,\cdots,x_n)$ is a legal weighting. For all $1\leq i<j\leq n$, we have
\begin{description}
  \item[(i)]  Suppose that $G$ does not cover the pair $\{i,j\}$. Then $\lambda(G)\leq\max\{\lambda(G-i),\lambda(G-j)\}$. In particular, $\lambda(G)\leq\lambda([n-1]^{(r)})$.
  \item[(ii)] Suppose that $m,t\in \mathbb{N}$ satisfy ${{t-1}\choose r}\leq m\leq {t\choose r}-{{t-2}\choose{r-2}}$. Then
  $$\lambda(C_{r,m})=\lambda([t-1]^{(r)})=\frac{1}{(t-1)^r}{{t-1}\choose r}.$$
  \item[(iii)] $w(G_i,\overset{\rightarrow}{x})\leq (1-x_i)^{r-1} \lambda(G_i)$ for any $i\in [n]$.
\end{description}
\end{lem}


Let $E\subset \mathbb{N}^{(r)}$, $e\in E$ and $i,j\in \mathbb{N}$ with $i<j$. We define
\begin{equation*}
L_{ij}(e)=\left\{\begin{array}{cc}
                             (e\less\{j\})\cup \{i\}, & \text{if}~i\notin e ~\text{and}~ j\in e;\\
                             e, & \text{otherwise},
                           \end{array}\right.
\end{equation*}
and $$\mathcal{C}_{ij}(E)=\{L_{ij}(e)\colon e\in E\}\cup \{e\colon e,L_{ij}(e)\in E\}.$$
We say that $E$ is {\it left-compressed} if $\mathcal{C}_{ij}(E)=E$ for every $1\leq i<j$.

From now on, we always assume ${{t-1}\choose r}\leq m< {t\choose r}$ for some integer $t$. Let $G$ be a graph with $e(G)=m$ which satisfies $\lambda(G)=\lambda_r(m)$ and let $\overset{\rightarrow}{x}$ be an (optimal) legal weighting attaining the Lagrangian of $G$. Without loss of generality, we can assume $x_i\geq x_j$ for all $i<j$ and $\overset{\rightarrow}{x}$ has the minimum possible number of non-zero entries, and let $T$ be this number.

Suppose that $G$ achieves a strictly larger Lagrangian than $C_{r,m}$. Then we have
$$\lambda(G)>\frac{1}{(t-1)^r}{{t-1}\choose r},$$
 which in turn implies $T\geq t$, otherwise $\lambda(G)\leq \lambda([t-1]^{(r)})$.

 \begin{lem}[\cite{FF1989,FR1984,T2002}]\label{four}
 Let $G$, $T$, and $\overset{\rightarrow}{x}$ be as defined above. Then
\begin{description}
\item[(i)] $G$ can be assumed to be left-compressed and to cover pairs.
\item[(ii)] For all $1\leq i\leq T$ we have  $$w(G_i,\overset{\rightarrow}{x})=r\lambda(G_i).$$
\item[(iii)] For all $1\leq i<j\leq T$ we have $$(x_i-x_j)w(G_{i,j},\overset{\rightarrow}{x})=w(G_{i\less j},\overset{\rightarrow}{x}).$$
\end{description}
\end{lem}

\begin{lem}[\cite{T2017}] \label{x1} Let $G$, $T$, and $\overset{\rightarrow}{x}$ be as defined above. Then for sufficiently large $t$,
\begin{description}
\item[(i)] $T=t+C$ for some constant $C=C(r)$.
\item[(ii)] $x_1<\frac{1}{t-\alpha}$ for some constant $\alpha=\alpha(r)>0$.
\item[(iii)] If $T=t$, then $x_1<\frac{1}{t-r+1}$.
\end{description}
\end{lem}

Here is our key lemma.

\begin{lem}\label{keylemma}
  Let $G$, $T$, and $\overset{\rightarrow}{x}$ be as defined above. There is a constant $t_0:=t_0(r)$ such that
  if $t\geq t_0$ then $T=t$.
\end{lem}

\section{Proof of Lemma \ref{keylemma}}
We need several lemmas before we prove Lemma \ref{keylemma}. 

Suppose that $G$ does not cover the pair $\{i,j\}$. Let $G_{/{ij}}$ be an $r$-uniform hypergraph
obtained from $G$ by gluing the vertices $i$ and $j$ as follows:
\begin{enumerate}
\item Let $v$ be a new fat vertex (by gluing $i$ and $j$.) Then
$G_{/ij}$ has the vertex set $(V(G)\setminus \{i,j\}) \cup \{v\}$.
\item The edge set of $G_{/ij}$ consists of all edges in $G$ not containing $i$ or $j$, plus
the edges of form $\{f\cup \{v\}\colon  f\in E_i \cup E_j\}$.
\end{enumerate}

The following lemma is similar to Lemma \ref{Hmr} (i), but has the advantage of being deterministic.
\begin{lem}\label{G/ij}
 Suppose that $G$ does not cover the pair $\{i,j\}$. Then $\lambda(G)\leq \lambda(G_{/ij})$.
\end{lem}

\begin{proof}
  Let $\vec x$ be an optimal legal weighting of $G$. Define a legal weighting $\vec y$ of $G_{/ij}$ by
$y_v=x_i+x_j$ and $y_k=x_k$ if $k\neq v$.
Then we have
  \begin{align*}
    w(G_{/ij},\vec y)- w(G, \vec x) &= \sum_{f\in E_i\cup E_j} y_fy_v - \sum_{f\in E_i}x_fx_i
- \sum_{f\in E_j}x_fx_j \\
&= \sum_{f\in E_i\cup E_j} x_f (x_i+x_j) - \sum_{f\in E_i}x_fx_i
- \sum_{f\in E_j}x_fx_j \\
&=  \sum_{f\in E_i\setminus E_j} x_fx_j + \sum_{f\in E_j\setminus E_i} x_fx_i\\
&\geq 0.
  \end{align*}
Then we have
$$\lambda(G)=w(G, \vec x)\leq  w(G_{/ij},\vec y)\leq \lambda(G_{/ij}).$$
\end{proof}

\begin{lem} \label{xk}
  For any $k\in[T-1]$, we have
  \begin{equation}\label{eq:xk}
   x_{T-k}>\frac{k-C_0}{k+1}x_1,
  \end{equation}
  where $C_0=C+\alpha-1$.
\end{lem}
\begin{proof}
Observe that
\begin{align*}
1&=x_1+\cdots+x_{T-k-1}+x_{T-k}+\cdots+x_T\\
&< (T-k-1)x_1+(k+1)x_{T-k}\\
&< \frac{T-k-1}{t-\alpha}+(k+1)x_{T-k}.~~~~~~(by~Lemma~ \ref{x1}~ (ii))
\end{align*}
Solving $x_{T-k}$ and applying Lemma \ref{x1} (i) and (ii), we have
$$x_{T-k}>\frac{t-\alpha-T+k+1}{k+1}\cdot\frac{1}{t-\alpha}=
\frac{k-C_0}{k+1}\cdot\frac{1}{t-\alpha}>\frac{k-C_0}{k+1}x_1.$$
\end{proof}

\begin{lem} \label{beta}
There exists a constant $\beta$ such that for any subset $S\subseteq [T]^{(r-2)}$, we have
  \begin{equation}
    \label{eq:sumxk}
    \sum_{f\in S}(x_1^{r-2}-x_f)<\beta x_1.
  \end{equation}
\end{lem}

\begin{proof}
Let $\beta$ be a constant $>\frac{C+\alpha}{(r-3)!}$.
 We will prove it by contradiction.
Suppose that there is $S\subseteq [T]^{(r-2)}$
such that $$\sum_{f\in S}(x_1^{r-2}-x_f)\geq\beta x_1.$$
We have
  \begin{equation}
    \label{eq:S}
  \sum_{f\in S}x_f \leq |S|x_1^{r-2}-\beta x_1.
  \end{equation}
Thus,
\begin{align*}
1&=(x_1+x_2+\cdots+x_T)^{r-2}\\
&\leq(r-2)!\sum_{f\in S}x_f+(T^{r-2}-(r-2)!|S|)x_1^{r-2}\\
&\leq (r-2)!(|S|x_1^{r-2}-\beta x_1) + (T^{r-2}-(r-2)!|S|)x_1^{r-2}~~~~~~~(by~\eqref{eq:S})\\
&= (T x_1)^{r-2}- (r-2)! \cdot \beta x_1.
\end{align*}
On the other hand, note 
$\frac{1}{t-\alpha}>x_1\geq \frac{1}{T}=\frac{1}{t+C}$. We have
\begin{align*}
  (T x_1)^{r-2}- (r-2)!\beta x_1
&<\left(1+\frac{C+\alpha}{t-\alpha}\right)^{r-2} - (r-2)!\cdot \beta \frac{1}{t+C}\\
&= 1+ (r-2)\frac{C+\alpha}{t-\alpha} +O(t^{-2}) - (r-2)!\cdot \beta \frac{1}{t-\alpha} +O(t^{-2})\\
&= 1- \frac{(r-2)!}{t-\alpha}\left(\beta- \frac{C+\alpha}{(r-3)!}\right)  +O(t^{-2})\\
&<1.
\end{align*}
Contradiction.
\end{proof}

Let $s=\max\{i\colon \{T-i-(r-2),\ldots,T-i-1,T-i,T\}\in E^c\}$ and $S=\{T-s,T-s+1,\ldots,T-1,T\}$.
We have the following lemma.
\begin{lem} \label{gamma} For $s$ and $S$ defined above, we have
  \begin{enumerate}
  \item Any non-edge $f\in E^c$ must intersect $S$ in at least two elements.
  \item  We have
    \begin{equation}
      \label{eq:XT-1}
    x_{T-1}\geq\gamma x_1,
    \end{equation}
    where $\gamma:=\prod_{k=0}^{r-2}\left(1-\frac{C_0+1}{s+k+1}\right)$.
  \end{enumerate}
\end{lem}
\begin{proof}
By the choice of $s$, we know $\{T-s-(r-2),\ldots,T-s-1,T-s,T\}\in E^c$
but $\{T-s-(r-1),\ldots,T-s-2,T-s-1,T\}\in E$. We now prove item 1 by contradiction.
Suppose not, say there is $f=\{i_1,i_2,\ldots, i_r\}\in E^c$ such that
$i_1<i_2<\cdots<i_{r-1}<T-s$. Since $G$ is left-compressed, $\{T-s-(r-1),\ldots,T-s-2,T-s-1,T\}\in E^c$. Contradiction!

Since  $G$ covers the pair $\{T-1,T\}$, there is an $(r-2)$-tuple
$\{i_1,i_2,\ldots,i_{r-2}\}$ $\in E_{T-1,T}$. We have
 \begin{align}\label{eq:x1xt-1xt}
x_{i_1}\cdots x_{i_{r-2}}x_{T-1}x_T-x_{T-s-(r-2)}x_{T-s-(r-3)}\cdots x_{T-s}x_T\geq 0.
\end{align}
Otherwise by replacing the edge $\{i_1,i_2,\ldots,i_{r-2},T-1,T\}$ with the non-edge $\{T-s-(r-2),T-s-(r-3),\ldots,T-s, T\}$,
we get another $r$-graph with the same number of edges whose Lagrangian is strictly greater than the Lagrangian of $G$. Contradiction!

Combining Inequalities \eqref{eq:x1xt-1xt} and \eqref{eq:xk}, we get
$$x_{T-1}\geq\frac{x_{T-s-(r-2)}x_{T-s-(r-3)}\cdots x_{T-s}}{x_1^{r-2}}> \gamma x_{1}.$$
\end{proof}

We have the following estimation of $\gamma$:
\begin{equation}
  \label{eq:gamma}
\gamma=\prod_{k=0}^{r-2}\left(1-\frac{C_0+1}{s+k+1}\right)=1-\frac{(C_0+1)(r-1)}{s}+ O\left(\frac{1}{s^2}\right).
\end{equation}

\begin{lem}\label{T>t}
  If $T>t$, then $|E^c_{T\backslash (T-1)}| -|E_{T-1,T}|\geq \frac{T-r}{T(r-1)}{T-2 \choose r-2}$.
\end{lem}
\begin{proof} Since $G$ is left-compressed, then
$$|E_T|\leq \frac{rm}{T}\leq \frac{r}{T}{t\choose r}=\frac{t}{T}{t-1\choose r-1}
\leq \frac{t}{T}{T-2\choose r-1}.$$
In the last step, we apply the assumption $t\leq T-1$.
Thus, we have
\begin{align*}
  |E^c_{T\backslash (T-1)}| -|E_{T-1,T}| &= {T-2 \choose r-1}  - (|E_{T\backslash (T-1)}| +|E_{T-1,T}|)\\
&={T-2\choose r-1} - |E_T|\\
&\geq {T-2\choose r-1} - \frac{t}{T}{T-2\choose r-1}\\
&=\frac{T-t}{T} {T-2\choose r-1}\\
&=\frac{(T-t)(T-r)}{(r-1)T}{T-2\choose r-2}\\
&\geq \frac{T-r}{T(r-1)}{T-2 \choose r-2}.
\end{align*}
\end{proof}

\begin{proof}[Proof of Lemma \ref{keylemma}:]


Let $G$, $\vec x$, $m$, $t$, and $T$ be as defined before. Let $\eta:=\left\lceil 2\sqrt {\frac{(C_0+1)(r-1)\beta}{(r-3)!}}t^{r-5/2}\right\rceil$.

We will prove Lemma \ref{keylemma} by contradiction. Assume $T>t$. By Lemma \ref{T>t}, we have
\begin{equation}
  \label{eq:F}
|E^c_{T\backslash (T-1)}| -|E_{T-1,T}|\geq \frac{T-r}{T(r-1)}{T-2 \choose r-2}>\eta.
\end{equation}
This is possible for $t$ sufficiently large since $\frac{T-r}{T(r-1)}{T-2 \choose r-2}=\Theta(t^{r-2})$
and $\eta=\Theta(t^{r-2.5})$.

Let $b=|E_{T-1,T}|$. Pick any $F\subseteq \{f\cup \{T\}\colon f\in E^c_{T\backslash (T-1)}\}$  of size $b+\eta$. This is possible
because of Inequality \eqref{eq:F}.

Let $G'$ be an $r$-graph obtained from $G$ by deleting all edges in $\{f\cup \{T-1,T\}\colon f\in E_{T-1,T}\}$
and adding all $r$-tuples in $F$ as edges.
The main proof is to show the following inequality:
\begin{equation}
  \label{eq:main}
  w(G,\vec x)\leq w(G',\vec x).
\end{equation}

Now we will prove Inequality \eqref{eq:main}. We divide it into two cases.

\noindent
{\bf Case 1:} $\beta s<\eta x_1^{r-3}$. By Lemma \ref{gamma} item 1, every non-edge intersects $S$ in at least two elements. This implies $F$ can be partitioned into $s-1$ parts:
$$F=\cup_{i=2}^s \{f\cup \{T-i, T\}\colon f\in F_i\},$$ 
where $F_i:=\{f\in [T-i-1]^{(r-2)}\colon f \cup \{T-i,T\} \in F\}$.
We have
\begin{align*}
  w(G',\vec x) &= w(G, \vec x) - \sum_{e\in E_{T-1,T}} x_ex_{T-1}x_T+ \sum_{f\in F} x_f\\
               &\geq w(G, \vec x)+\eta x_1^{r-2}x_{T-1}x_T -\sum_{f\in F}(x_1^{r-2}x_{T-1}x_T-x_f)\\
               &\geq w(G, \vec x)+\eta x_1^{r-2}x_{T-1}x_T - \sum_{i=2}^s
                 \sum_{f'\in F_i}(x_1^{r-2}-  x_{f'})  x_{T-1}x_T\\
                &\geq w(G, \vec x)+\eta x_1^{r-2}x_{T-1}x_T - x_{T-1}x_T\sum_{i=2}^s \beta x_1   ~~~~~~~~~~~~~~~~~~~(\mbox{by}~ \eqref{eq:sumxk})\\
                &>w(G, \vec x)+ \left(\eta x_1^{r-3}-s\beta\right)x_1x_{T-1}x_T\\
                &>w(G, \vec x).
\end{align*}

\noindent
{\bf Case 2:} $\beta s\geq \eta x_1^{r-3}$.  Since $x_1\geq \frac{1}{T}$, we have
\begin{equation}
  \label{eq:s}
s\geq \frac{\eta}{\beta} x_1^{r-3}\geq \frac{\eta}{\beta} \frac{1}{T^{r-3}}.
\end{equation}

We first prove an inequality:
\begin{equation}
  \label{eq:gammab}
   \gamma^{r-2}\eta > (1-\gamma^{r-2})b.
\end{equation}
\begin{align*}
&\hspace*{-5mm}  \gamma^{r-2}\eta - (1-\gamma^{r-2})b\\
  &\geq  \gamma^{r-2}\eta - (1-\gamma^{r-2}){T-2\choose r-2}\\
  &>\gamma^{r-2} \left( \eta- (\gamma^{-(r-2)} -1) \frac{T^{r-2}}{(r-2)!} \right)\\
  &>\gamma^{r-2} \left( \eta- \left(\frac{(C_0+1)(r-1)(r-2)}{s} + O\left(\frac{1}{s^2}\right) \right)
    \frac{T^{r-2}}{(r-2)!} \right)
    ~~~~~(\mbox{by}~ \eqref{eq:gamma})\\
  &=\frac{\gamma^{r-2}}{s}  \left( s\eta- \left((C_0+1)(r-1)(r-2) + O\left(\frac{1}{s}\right) \right)
    \frac{T^{r-2}}{(r-2)!} \right)~~(\mbox{using}~\eqref{eq:s})\\
  &\geq\frac{\gamma^{r-2}}{s}  \left(\frac{1}{\beta T^{r-3}}\eta^2- \left((C_0+1)(r-1)(r-2) + O\left(\frac{1}{s}\right) \right)
    \frac{T^{r-2}}{(r-2)!} \right)\\
  &>0.
\end{align*}
In the last step, we apply the definition of $\eta $ and get
$$\eta^2\geq 4 \frac{(C_0+1)(r-1)(r-2)\beta}{(r-2)!}t^{2r-5}>
2 \frac{(C_0+1)(r-1)(r-2)\beta}{(r-2)!} T^{2r-5}.$$
Now we are ready to estimate $ w(G',\vec x)$:
\begin{align*}
  w(G',\vec x)  &= w(G, \vec x) - \sum_{e\in E_{T-1,T}} x_ex_{T-1}x_T+ \sum_{f\in F} x_f\\
                &>w(G, \vec x)-b x_1^{r-2}x_{T-1}x_T + (b+\eta) x_{T-1}^{r-1}x_T\\
                &=w(G, \vec x)+x_{T-1}x_T\left((b+\eta) x_{T-1}^{r-2}-b x_1^{r-2}\right)\\
                &\geq w(G, \vec x)+ x_{T-1}x_T\left((b+\eta) \gamma^{r-2} x_1^{r-2}- b x_1^{r-2}\right)~~~~~~(by~\eqref{eq:XT-1})\\
                &=w(G, \vec x)+ x_1^{r-2}x_{T-1}x_T\left(\gamma^{r-2}\eta -(1-\gamma^{r-2})b\right)~~~~~~(by~\eqref{eq:gammab})\\
               &> w(G, \vec x).
\end{align*}
Therefore, Inequality \eqref{eq:main} holds in any circumstances.

Note that  $G'$ does not cover the pair $\{T-1, T\}$.
Applying Lemma \ref{G/ij}, we have
$$\lambda(G)=w(G, \vec x)<  w(G',\vec x)\leq \lambda(G')\leq \lambda(G'_{/(T-1)T}).$$
By the construction of $G'$, the added edges  are from $F\subseteq \{f\cup \{T\}\colon f\in E^c_{T\backslash (T-1)}\}$.
These edges have the form of $f\cup \{T\}$, where $f\cup \{T-1\}$ is also an edge in $G$.
After gluing $T-1$ and $T$ together, both edges $f\cup \{T\}$ and $f\cup \{T-1\}$ are glued into one edge $f\cup \{v\}$. We have
$$|E(G'_{/(T-1)T})|\leq |E(G')|-|F|\leq |E(G)|.$$
Contradiction! This completes the proof of Lemma~\ref{keylemma}.

\end{proof}

\section{Proof of Theorem \ref{mainthm}}
Assume $t\geq t_0$.  Let $G=(V,E)$ be an $r$-graph with $m$ edges satisfying ${{t-1}\choose r}\leq m\leq {t\choose r}-{t-2\choose r-2}$ and $\lambda(G)=\lambda_r(m)>\lambda([t-1]^{(r)})$.
Let $\vec x=(x_1,\ldots, x_n)$ be an optimal legal weighting for $G$ that uses exactly $T$ nonzero weights (i.e., $x_1\geq \cdots \geq x_T> x_{T+1}=\cdots =x_n=0$).  By Lemma \ref{keylemma}, we may assume $T=t$. In addition, we may assume $G$ is left-compressed
by Lemma \ref{four}(i).

Since $T=t$, by Lemma \ref{x1}(iii), we have
\begin{equation}
  \label{t:x1}
x_1<\frac{1}{t-r+1},
\end{equation}

and we may improve Lemma \ref{xk} as follows.
\begin{lem} \label{xkt}
  For any $k\in[t-1]$, we have
  \begin{equation}\label{eq:xkt}
   x_{t-k}>\frac{k-r+2}{k+1}x_1.
  \end{equation}
\end{lem}
Let $T=t$, $C=0$, $\alpha=r-1$, $C_0=r-2$, and $\beta=\frac{r}{(r-3)!}$ in the proof of Lemma \ref{beta}. Then we can get the following Lemma improving Lemma \ref{beta}.
\begin{lem} \label{betat}
For any subset $S\subseteq [t]^{(r-2)}$, we have
  \begin{equation}
   \label{t:beta}
    \sum_{f\in S} (x_1^{r-2}-x_f)<\frac{r}{(r-3)!} x_1.
  \end{equation}
\end{lem}

With $T=t$, we have $s=\max\{i\colon \{t-i-(r-2),\ldots,t-i-1,t-i,t\}\in E^c\}$ and $S=\{t-s,t-s+1,\ldots,t-1,t\}$. Lemma \ref{gamma} can be improved to:
\begin{lem} \label{gammat} For $s$ and $S$ defined above, we have
  \begin{enumerate}
  \item Any non-edge $f\in E^c$ must intersect $S$ in at least two elements.
  \item  We have
    \begin{equation}
      \label{eq:xt-1}
    x_{t-1}>\gamma x_1,
    \end{equation}
    where $\gamma:=\prod_{k=0}^{r-2}\left(1- \frac{r-1}{s+k+1}\right)$.
  \end{enumerate}
\end{lem}

\begin{lem} \label{minimum}
Let $\gamma$ be defined in Lemma \ref{gammat}. Consider two functions:
  \begin{align}
    g_1(s)&:={s+1\choose 2}\frac{r\gamma^{-(r-2)}}{(r-3)!}t^{r-3},\\
    g_2(s)&:={t-2\choose r-2}(\gamma^{-(r-2)}-1).
  \end{align}  
  Then
  \begin{equation}
  \label{eq:eta}
  \min\{g_1(s),g_2(s)\}=O(t^{r-\frac7 3}).
\end{equation}
\end{lem}
\begin{proof}
 By Equation \eqref{eq:gamma} (with $C_0=r-2$),   we have
  $$\gamma=1-\frac{(r-1)^2}{s}+O\left(\frac{1}{s^2}\right).$$  
  When $s\leq t^{1/3}$, we have
  $$g_1(s)={s+1\choose 2}\frac{r\gamma^{-(r-2)}}{(r-3)!}t^{r-3} =O(t^{r-7/3}).$$
  When $s\geq t^{1/3}$, we have
  \begin{align*}
    g_2(s)&={t-2\choose r-2}(\gamma^{-(r-2)}-1)\\
          &= {t-2\choose r-2}\left(\frac{(r-1)^2(r-2)}{s}+O\left(\frac{1}{s^2}\right)\right)\\
            &=O(t^{r-7/3}).
  \end{align*}
  Thus, Equality \eqref{eq:eta} holds.   
\end{proof}

\begin{proof}[Proof of Theorem \ref{mainthm}:]
  We assume $t\geq t_0$ so that $T=t$ holds. Let $\eta:=\lceil \min\{f(s),g(s)\}\rceil =O(t^{r-7/3})$.
  
\noindent {\bf Claim\refstepcounter{counter}\label{mainclaim}  \arabic{counter}.} For  $t\geq t_0$, if $ {t-1\choose r}\leq m\leq {t\choose r}-{t-2\choose r-2}-\eta$, then $\lambda_r(m)=\frac{{{t-1}\choose r}}{(t-1)^r}$.

We will prove this claim by contradiction. Assume there is a graph $G$ with $m$ edges, where $ {t-1\choose r}\leq m\leq {t\choose r}-{t-2\choose r-2}-\eta$, and $\lambda(G)=\lambda_r(m)>\frac{{{t-1}\choose r}}{(t-1)^r}$.
Let $B$ be any family of $|E_{t-1,t}|$ many non-edges of $G$ which does not contain both $t$ and $t-1$.
This is possible since $G$ has at least ${{t-2}\choose{r-2}}+\eta$ non-edges.
Let $G'$ be an $r$-graph obtained from $G$ by deleting all edges in $E_{t-1,t}$
and adding all $r$-tuples in $B$ as edges. Then $G'$ has exactly $m$ edges.

By Lemma \ref{gammat} item 1, any non-edge in $B$ must intersect $S$ in at least two elements.
For any $\{i,j\}\subseteq S^{(2)}$ with $i<j$, define
$$B_{ij}:=\{\{i_1,\ldots,i_{r-2}\}
\colon \{i_1,\ldots,i_{r-2},i,j\}\in B
~\text{and}~ i_1<\cdots<i_{r-2}<i<j\}.$$
Then we have $$B= \bigcup\limits_{\{i,j\}\subseteq S^{(2)}}\{f\cup \{i,j\}\colon f\in B_{ij}\}.$$
We allow some $B_{ij}$ to be emptysets. (For example, $B_{t-1,t}=\emptyset$.) 

Now, we consider the difference between $w(G',\vec x)$ and $w(G, \vec x)$.
On the one hand,
\begin{align*}
  w(G',\vec x) &= w(G, \vec x) - \sum_{e\in E_{t-1,t}} x_ex_{t-1}x_t+ \sum_{f\in B} x_f\\
               &\geq w(G, \vec x)- \sum_{f\in B}(x_1^{r-2}x_{t-1}x_t-x_f)\\
               &\geq w(G, \vec x)-\sum_{\{i,j\}\in S^{(2)}}\sum_{f'\in B_{ij}}(x_1^{r-2}-x_{f'})x_{t-1}x_t\\
                &>w(G, \vec x)- {s+1\choose 2}\frac{r}{(r-3)!}x_1x_{t-1}x_t ~~~~~~~~~~~~~~~~~~~~~~~~(by~ \eqref{t:beta})\\
                &>w(G, \vec x)- {s+1\choose 2}\frac{r}{(r-3)!}t^{r-3}x_{1}^{r-2}x_{t-1}x_t.  ~~~~~~~~(\mbox{since } x_1\geq \frac{1}{t})\\
               &>w(G, \vec x)- {s+1\choose 2}\frac{r\gamma^{-(r-2)}}{(r-3)!}t^{r-3}x_{t-1}^{r-1}x_t.  ~~~~~~~~~~~~~~~~~~~(by~ \eqref{eq:xt-1})\\
  &=w(G, \vec x)- g_1(s) x_{t-1}^{r-1}x_t.
\end{align*}

On the other hand,
 \begin{align*}
   w(G',\vec x) &= w(G, \vec x) - \sum_{e\in E_{t-1,t}} x_ex_{t-1}x_t+ \sum_{f\in B} x_f\\
    &> w(G, \vec x)- \sum_{f\in B}(x_1^{r-2}x_{t-1}x_t-x_{t-1}^{r-2}x_{t-1}x_t)\\
               &> w(G, \vec x)-{t-2\choose r-2}(x_1^{r-2}-x_{t-1}^{r-2})x_{t-1}x_t\\
               &=w(G, \vec x)- {t-2\choose r-2}(\gamma^{-(r-2)}-1)x_{t-1}^{r-1}x_t ~~~~~~~~~~~~~~~~~~~~~~~~~~~~(by~ \eqref{eq:xt-1})\\
    &=w(G, \vec x)- g_2(s) x_{t-1}^{r-1}x_t.
\end{align*}
Thus, we have
\begin{equation}
  \label{eq:diff}
   w(G',\vec x) >  w(G, \vec x) -\min\{g_1(s), g_2(s)\} x_{t-1}^{r-1}x_t \geq w(G, \vec x) - \eta x_{t-1}^{r-1}x_t.
\end{equation}

Note that $G'$ does not cover $\{t-1,t\}$ and $G'$ still has $\eta$ non-edges which does not contain
both $t-1$ and $t$. Let $G''$ be an $r$-graph obtained from $G'$ by adding these $\eta$ $r$-tuples
as edges. We have
\begin{align*}
  w(G'',\vec x) &\geq w(G',\vec x) +\eta x_{t-r+1}x_{t-r+2}\cdots x_t\\
                &> w(G, \vec x) - \eta x_{t-1}^{r-1}x_t + \eta x_{t-r+1}x_{t-r+2}\cdots x_t\\
                &>w(G, \vec x).
\end{align*}
Note that $G''$ still does not cover the pair $\{t-1,t\}$. We have
 $$\lambda(G)=w(G,\vec x) <w(G'',\vec x)\leq \lambda(G'')\leq\lambda([t-1]^{(r)}),$$
a contradiction. Claim \ref{mainclaim} is proved.

Finally we can choose a constant $\delta_r$ large enough so that the following two conditions hold:
\begin{itemize}
\item $\delta_r t^{r-7/3}>\eta$ for all $t\geq t_0$,
\item and $\delta_r t^{r-7/3}>{t-2 \choose r-1}$ for $1\leq t\leq t_0.$
\end{itemize}
When $t\geq t_0$, we have
$$\lambda_r(m)\leq \lambda\left(C_{r, {t\choose r}-{t-2\choose r-2}-\delta_rt^{r-7/3}}\right) \leq \lambda\left(C_{r,{t\choose r}-{t-2\choose r-2}-\eta}\right)=\frac{{t-1\choose r}}{(t-1)^r}.$$
When $1\leq t\leq t_0$, we have
$$\lambda_r(m)\leq \lambda\left(C_{r,{t-1\choose r}}\right)= \frac{{t-1\choose r}}{(t-1)^r}.$$
Since $m\geq {t-1\choose r}$, we have
$$\lambda_r(m)\geq \lambda_r\left({t-1\choose r}\right)=\frac{{t-1\choose r}}{(t-1)^r}.$$
Thus, $$\lambda_r(m)=\frac{{t-1\choose r}}{(t-1)^r}.$$
This completes the proof of Theorem~\ref{mainthm}.

  \end{proof}

\end{document}